\documentclass[12pt,a4paper]{article}
\usepackage[cp1250]{inputenc}

\date{\today}
\usepackage{amsmath}
\usepackage{amssymb}
\usepackage{amsthm}
\usepackage{esint}

\usepackage{natbib}
\usepackage{amsfonts}
\usepackage{amsthm}
\usepackage{mathtools}
\usepackage{float}
\usepackage[usenames]{color}
\usepackage[T1]{fontenc}
\usepackage{graphicx}
\usepackage{caption}
\usepackage{algpseudocode}
\usepackage{setspace}

\newcounter{thm}
\newtheorem{lem}[thm]{Lemma}
\newtheorem{cor}[thm]{Corollary}

\newtheorem{theorem}[thm]{Theorem}

\theoremstyle{definition}
\newtheorem{df}[thm]{Definition}

\newcounter{rys}
\newcommand{\p}{\partial}

\renewcommand{\d}{\mathrm{d}}

\newcommand{\eqnb}{\begin{equation}}
\newcommand{\eqnbs}{\begin{equation*}}
\newcommand{\eqnbsa}{\begin{equation*}\begin{aligned}}
\newcommand{\eqnba}{\begin{equation}\begin{aligned}}
\newcommand{\eqnbl}[1]{\begin{equation}\label{#1}}
\newcommand{\eqnbal}[1]{\begin{equation}\label{#1}\begin{aligned}}
\newcommand{\eqnes}{\end{equation*}}
\newcommand{\eqne}{\end{equation}}
\newcommand{\eqnesa}{\end{aligned}\end{equation*}}
\newcommand{\eqnea}{\end{aligned}\end{equation}}

\newcommand{\comment}[1]{}
\newcommand{\RR}{\mathbb{R}}

\begin{document}
\setcounter{rys}{0}

%\subsection{Preface}

\begin{center}
 \begin{spacing}{1.3}
     \Large\textbf{A generalised comparison principle for the Monge--Amp\`ere equation and the pressure in 2D fluid flows}
  \end{spacing}
  Wojciech S. Ozanski
  \end{center}
 \begin{center}
\Large{\textbf{Abstract}}
\end{center}
\nocite{larcheveque}
\nocite{larcheveque1}
We extend the \emph{generalised comparison principle} for the Monge--Amp\`ere equation due to Rauch \& Taylor (\emph{Rocky Mountain J. Math.} 7, 1977) to nonconvex domains. From the generalised comparison principle we deduce bounds (from above and below) on solutions of the Monge--Amp\`ere equation with sign-changing right-hand side. As a consequence, if the right-hand side is nonpositive (and does not vanish almost everywhere) then the equation equipped with constant boundary condition has no solutions. In particular, due to a connection between the two-dimensional Navier--Stokes equations and the Monge-Amp\`ere equation, the pressure $p$ in 2D Navier-Stokes equations on a bounded domain cannot satisfy $\Delta p \leq 0$ in $\Omega $ unless $\Delta p \equiv 0$ (at any fixed time). As a result at any time $t>0$ there exists $z\in \Omega $ such that $\Delta p (z,t) =0$. 

\setcounter{page}{1}
%\begin{center}
% \includegraphics[width=7cm]{intro.eps}
% \nopagebreak
%  \captionof{figure}{The setup.}\label{intro} 
%\end{center}

\section{Introduction}\label{motivation_section}
The Monge-Amp\`ere equation is
\[
\mathrm{det}\, D^2\phi=f.
\]
When the right-hand side $f$ of this equation is positive it constitutes an example of a nonlinear second order elliptic equation (see, for example, Chapter 17 of \cite{gilbarg_trudinger}) and the study of the Dirichlet boundary value problem for this equation goes back to the works of \cite{alexandrov_58}, \cite{bakelman_57} and \cite{pogorelov_ma} and it is related to the prescribed Gaussian curvature problem and the Monge-Kantorovich mass transfer problem. Important results in this theory include Alexandrov maximum principle (Aleksandrov, 1968\nocite{aleksandrov_68}), the equivalence between the notion of the generalised solution and the notion of viscosity solution (Caffarelli, 1989\nocite{caffarelli_89}), and, most notably, the interior regularity results (see Caffarelli, 1990\nocite{caffarelli_w2p_estimates}). See \cite{monge_ampere} for a modern exposition of the theory of the Monge-Amp\`ere equation.

Moreover the Monge-Amp\`ere equation with positive right-hand side $f$ shares many striking similarities with the Laplace equation; take for instance the fact that both the Laplace operator $\Delta \phi$ and the determinant of the Hessian $\mathrm{det}\, D^2 \phi$ are invariant under orthogonal transformations, the similarity between the comparison principle (see Corollary \ref{comparison_cor}) and the maximum principle for subharmonic functions, or the occurence of Perron's method in finding solutions to the Dirichlet boundary value problem.

However, very little is known about the Monge-Amp\`ere equation when the right-hand side $f$ changes sign since in this case it is a (nonlinear) mixed elliptic-hyperbolic problem. A step in this direction is a generalised comparison principle (see Theorem \ref{gen_comp_princ}), first studied by \cite{rauch_taylor} in the case of strictly convex domain $\Omega$. Their study was partially motivated by the problem of the buckling thin elastic shell, which gives rise to the homogeneous Monge--Amp\`ere equation $M\phi=0$ in $\Omega$ equipped with Dirichlet boundary condition $\left. \phi \right|_{\partial \Omega}=g$. The generalised comparison principle enabled them to establish bounds (above and below) on a solution to this problem, which in the case $g\equiv 0$ established uniqueness of the solution $\phi\equiv 0$ to the problem $M\phi =0$ in $\Omega$, $\left. \phi \right|_{\partial \Omega} =0$ (the standard standard existence and uniqueness theorem (see Theorem \ref{existence_thm}) gives uniqueness only among convex solutions). This result filled a gap in the uniqueness problem of the elastic shell (see the first section of \cite{rauch_taylor} and Remark 2.2 in \cite{rabinowitz}).
Moreover, thanks to the generalised comparison principle one can deduce pointwise bounds, above and below, to the solution of the Monge--Amp\`ere equation with a sign-changing right-hand side (see Corollary \ref{mainthm}).

Here we further extend this comparison principle to cover the case of nonconvex domains $\Omega$ and we point out an interesting application in the theory of the two-dimensional Navier--Stokes equations.

In the next section we recall some background theory of the Monge--Amp\`{e}re equation. In Section \ref{gen_comp_princ_section} we prove the generalised comparison principle and discuss its consequences (bounds on the solution of the Monge--Amp\`ere equation). In the last section (Section \ref{application_NS_section}) we discuss the link between the two-dimensional Navier--Stokes equations and the Monge-Amp\`ere equation and we use the bounds on solution to the Monge-Amp\`ere equation to show that the pressure $p$ in 2D Navier-Stokes equations on a bounded domain cannot satisfy $\Delta p\leq 0$, $\Delta p\not \equiv 0$ at any $t>0$.
\section{Preliminary material}
Let $\Omega $ be a bounded, open subset of $\mathbb{R}^n$. We will use a number of properties of convex functions (and concave functions), the Monge--Amp\`ere measure and the Monge--Amp\`ere equation. In this section we quickly recall the relevant definitions and results; the proofs can be found in the first chapter of \cite{monge_ampere}.\\

A function $\phi\colon \Omega \to \RR$ is convex if $\phi(\lambda x + (1-\lambda )y) \leq \lambda \phi(x) + (1-\lambda )\phi (y)$ for every segment $[x,y]\subset \Omega$ and $\lambda \in [0,1]$. If $\phi \in C^2(\overline{\Omega })$ then $\phi$ is convex in $\Omega$ if and only if $D^2 \phi$ is positive definite in $\Omega$. A set $\Omega$ is \emph{convex} if $\lambda x+(1-\lambda )y \in \overline{\Omega }$ for all $x,y \in \overline{\Omega }$, $\lambda \in [0,1]$; it is \emph{strictly convex} if $\lambda x+(1-\lambda )y \in {\Omega }$ for all $x,y \in \overline{\Omega }$, $\lambda \in (0,1)$. A \emph{supporting hyperplane} If for some for $\phi$ at $x_0\in \Omega$ is an affine function $\phi(x_0) + m\cdot (x-x_0)$ such that 
\[
\phi (x) \geq \phi (x_0) + m\cdot (x-x_0) \quad \text{ for all } x\in \Omega . 
\]
\begin{df}
The \emph{normal mapping} of $\phi$ (or \emph{subdifferential} of $\phi$) is the set-valued mapping $\partial \phi : \Omega  \to \mathcal{P} (\mathbb{R}^n )$, which maps $x_0\in \Omega$ into the set of all those $m$ for which $\phi(x_0)+m\cdot (x-x_0)$ is a supporting hyperplane. Namely
\[
\partial \phi (x_0) \coloneqq \{ m \in \mathbb{R}^n \, : \, \phi(x) \geq \phi(x_0) + m\cdot (x-x_0) \quad  \text{ for all } x \in U_{x_0} \},
\]
where $U_{x_0}$ denotes some open neighbourhood of $x_0$. Given $E \subset \Omega $ we define 
\[ \partial \phi (E) = \cup_{x\in E } \, \partial \phi (x).\]
\end{df}
A convex function $\phi$ has at least one supporting hyperplane at each point, that is $\p \phi (x_0) \ne \emptyset$ for all $x_0\in \Omega$. If $\phi \in C(\overline{\Omega })$ then the family of sets 
\[ \mathcal{S}:= \{ E \subset \Omega \, : \, \partial \phi (E)\text{ is Lebesgue measurable} \}
\]
is a Borel $\sigma$-algebra.
\begin{df}
The set function $M\phi: \mathcal{S} \to \overline{\mathbb{R}}$ defined by 
\[ M\phi (E) := | \partial \phi (E) | ,\]
where $\overline{\mathbb{R}} := \mathbb{R} \cup \{ \infty \}$ is the \emph{Monge--Amp\`ere measure} of $\phi$. 
\end{df}
In a sense, $M\phi (E)$ measures ``how convex'' $\phi$ is on $E$. Moreover this measure is finite on compact subsets of $\Omega$ and it satisfies the following three properties.
\begin{enumerate}
\item[(i)] It can be verified using Sard's Theorem that if $\phi\in C^2 (\overline{ \Omega })$ then $M\phi$ is absolutely continuous with respect to the Lebesgue measure and $M\phi (E) = \int_E \mathrm{det}\, D^2 \phi \, \mathrm{d}x$ for all Borel sets $E\subset \Omega $. In particular, if $\phi(x) := \delta |x-x_0|^2$ for some ${x_0\in \Omega }$, $\delta >0$ then $D^2 \phi =2\delta I$, where $I$ denotes the unit matrix, and so 
\[ M\phi (E) = \int_E \mathrm{det}\, D^2 \phi \, \mathrm{d}x = (2\delta )^n |E| \]
for every Borel set $E$.
\item[(ii)] If $\phi,\psi$ are convex functions then $M(\psi+\phi)  \geq M\psi + M\phi$. In particular, adding a constant function has no effect on the Monge--Amp\`ere measure. On the other hand, adding a quadratic polynomial $\delta |x-x_0|^2$ to a convex function $\psi $ strictly increases its Monge--Amp\`{e}re measure, that is
\[\widetilde{\psi } := \psi + \delta |\cdot - x_0|^2 \quad \Rightarrow \quad M\widetilde{\psi} (E) \geq M\psi (E) +  (2\delta )^n |E| \quad \]
for every Borel $E$.
\item[(iii)] If $\Omega \subset \RR^n$ is a bounded open set and $\phi, \psi \in C(\overline{\Omega })$ are such that $\phi = \psi $ on $\partial \Omega$ with $\psi \leq \phi $ then $\partial \phi (\Omega ) \subset \partial \psi (\Omega )$ and hence also
\[
M \phi (\Omega ) \leq M\psi (\Omega ) ,
\]
see Figure \ref{max_princ_figure}. The proof of this property is in fact an algebraic translation of the figure, see \cite{monge_ampere}, pp. 10-11.
\end{enumerate}
\begin{center}
 \includegraphics[width=\textwidth]{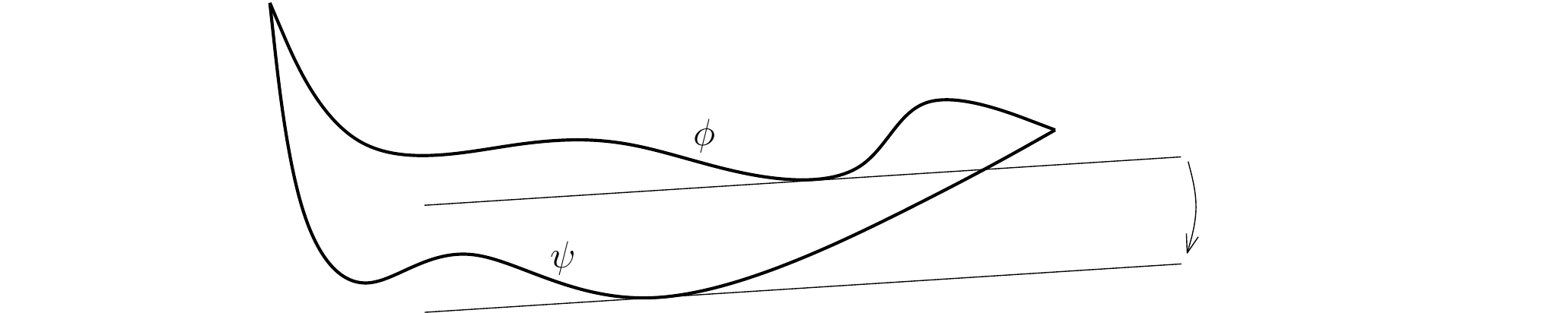}
 \nopagebreak
  \captionof{figure}{$\partial \varphi (\Omega ) \subset \partial \phi (\Omega )$ (based on Fig. 1.2 from \cite{monge_ampere}; note this is a 1D sketch of a multidimensional situation).}\label{max_princ_figure} 
\end{center}
The comparison principle (Corollary \ref{comparison_cor}) is an important tool in studying the Monge--Amp\`ere equation. Here we present a stronger version of the comparison principle. We focus on the case of convex functions; the case of concave functions follows analogously by replacing, respectively,  $M\phi$, $M\psi$ by $M(-\phi)$, $M(-\psi)$ and ``minimum'' by ``maximum''.
\begin{theorem}\label{comparison_thm}{\normalfont (Strong comparison principle)} Let $\Omega $ be open and bounded and ${\phi,\psi \in C(\overline{\Omega } )}$ be convex functions such that
\begin{equation}\label{comparison_assumption}
M\phi  \leq M\psi  \qquad \text{ on } \Omega .
\end{equation}
If $\widetilde{\psi } \coloneqq \psi + Q$ for some quadratic polynomial $Q(x) \coloneqq \delta |x-x_0|^2$ then $\phi - \widetilde{\psi }$ does not attain its minimum inside $\Omega$.
\end{theorem}
This theorem will be important in obtaining our generalised comparison principle for nonconvex domains (see  Theorem \ref{gen_comp_princ}). We prove it by sharpening the proof of the standard comparison principle, see e.g. \cite{monge_ampere}, pp. 16-17.
\begin{proof}
Suppose that there exists $z \in \Omega$ such that
\[
\phi (z)-\widetilde{\psi} (z) = \min_{\overline{\Omega }} (\phi - \widetilde{\psi}) =: a
\]
and let 
\[\widetilde{Q}(x) \coloneqq \frac{\delta }{2} | x - (2x_0 -z )|^2 - \delta |z-x_0|^2 .
\]
This quadratic polynomial is tangent to $Q(x)$ at $z$ and supports it from below, that is $\widetilde{Q}(z)=Q(z)$ and $\widetilde{Q}(x)<Q(x)$ for $x\ne z$. Indeed, direct calculation gives \mbox{$\widetilde{Q}(z)=Q(z)$}, $\nabla \widetilde{Q} (z) = \nabla Q(z)$, $D^2(Q-\widetilde{Q})=\delta I$ and so Taylor's expansion for $x\ne z$ gives
\[
\left( Q - \widetilde{Q} \right) (x) = (x-z) \cdot \frac{\delta I}{2} (x-z) = \frac{\delta}{2} | x-z|^2 >0.
\]
Hence in particular $\left. \widetilde{Q} \right|_{\partial \Omega } < \left. {Q} \right|_{\partial \Omega }$ and we obtain 
\[ b\coloneqq \min_{\partial \Omega } \left( \phi - \psi - \widetilde{Q} \right) > \min_{\partial \Omega } \left( \phi - \psi - {Q} \right) \geq a,
\]
see Figure \ref{strong_comp_princ_figure}. Now let
\[
w(x) \coloneqq \psi (x) + \widetilde{Q} (x) + \frac{b+a}{2}
\]
and
\[
G \coloneqq \{ x\in \Omega \, :\, \phi (x) < w (x) \}.
\]
We see that $z \in G$ and so $G$ is a nonempty open subset of $\Omega$. 
\begin{center}
 \includegraphics[width=0.6\textwidth]{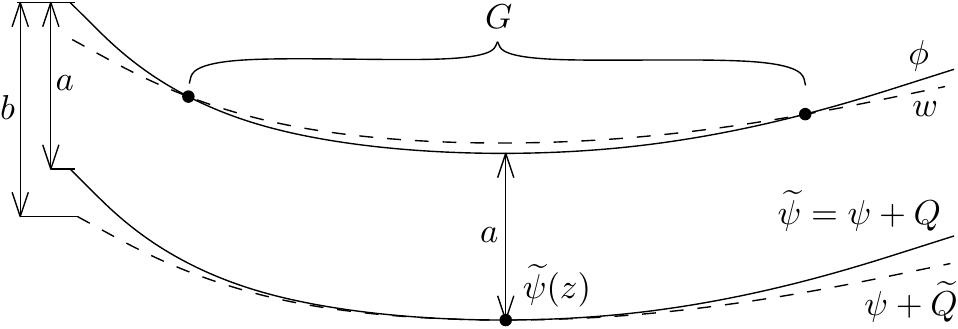}
 \nopagebreak
  \captionof{figure}{The set $G$ (note this is a 1D sketch of a multidimensional situation).}\label{strong_comp_princ_figure} 
  \vspace{0.2cm}
\end{center}
Hence $|G|>0$ and property (ii) gives
\eqnb\label{temp32}
M\psi (G) + \delta^n |G| \leq Mw (G). 
\eqne
Moreover $\partial G = \{ x \in \Omega \, : \, w(x) = \phi (x) \} $ (see Figure \ref{strong_comp_princ_figure}). Indeed, this is equivalent to  $\overline{G} \cap \partial \Omega = \emptyset$, but for $y\in \partial \Omega$ we have
\[
\phi (y) - \psi (y) - \widetilde{Q} (y) \geq b > \frac{b+a}{2},
\]
that is $\phi (y) > w(y)$ and so $y\not \in \overline{G}$. Therefore indeed $\partial G = \{ x \in \Omega \, : \, w(x) = \phi (x) \} $ and hence property (iii) gives 
\[
 M w (G) \leq M \phi (G).
\]  
This and \eqref{temp32} gives $M\psi (G) < M \phi (G)$, which contradicts the assumption \eqref{comparison_assumption}. 
\end{proof}
The standard comparison principle (Theorem 1.4.6 in \cite{monge_ampere}) is a corollary of Theorem \ref{comparison_thm}.
\begin{cor}\label{comparison_cor}{\normalfont (Comparison principle)} Let $\Omega $ be open and bounded and $\phi,\psi \in C(\overline{\Omega } )$ be convex functions such that $
M\phi  \leq M\psi$ in $\Omega $. Then
\[
\min_{x\in \overline{\Omega } } (\phi-\psi) (x) =\min_{x\in \partial \Omega  } (\phi-\psi) (x)  .
\]
In particular, if $\phi \geq \psi$ on $\partial \Omega$ then $\phi \geq \psi$ in $\overline{\Omega }$.
\end{cor}
\begin{proof}
Suppose otherwise that there exists an $x_0 \in \Omega$ such that
\[
\phi (x_0) - \psi (x_0) =  \min_{x\in \overline{\Omega } } (\phi-\psi) (x) < \min_{x\in \partial \Omega  } (\phi-\psi) (x) .
\]
Because $\Omega $ is bounded, for sufficiently small $\delta >0$ the function $\phi-(\psi +\delta |x-x_0|^2)$ still attains its minimum inside $\Omega$, which contradicts the strong comparison principle (Theorem \ref{comparison_thm}).
\end{proof}
If $\mu$ is a Borel measure defined in $\Omega$ we say that a convex function $v\in C(\Omega )$ is a \emph{generalised solution} to the Monge--Amp\`{e}re equation $\mathrm{det}\, D^2 v = \mu$ if $Mv=\mu$. If $v$ is concave, it is a generalised solution of $\mathrm{det}\, D^2 v = \mu$ when $M(-v)=\mu$. We have the following existence and uniqueness result for the Dirichlet problem for the Monge--Amp\`ere equation (Theorem 1.6.2 in \cite{monge_ampere}).
\begin{theorem}\label{existence_thm}{\normalfont (Existence theorem for the Monge--Amp\`{e}re equation)}
If $\Omega \subset \mathbb{R}^n$ is open, bounded and strictly convex, $\mu$ is a Borel measure in $\Omega$ with $\mu ( \Omega ) < +\infty $ and $g\in C(\partial \Omega )$, then there exists a unique convex generalised solution $\psi \in C(\overline{\Omega})$ to the problem \[ \begin{cases}
\mathrm{det}\, D^2 \psi = \mu \qquad & \text{ in } \Omega , \\
\psi =g &\text{ on } \partial \Omega .
\end{cases}\]
Similarly there exists a unique concave generalised solution to this problem.
\end{theorem}
Before turning to the generalised comparison principle we recall the following weak convergence result for Monge--Amp\`ere measures.
\begin{lem}\label{conv_lemma_2}
Let $\Omega \subset \mathbb{R}^n$ be an open, bounded and strictly convex domain, $\mu_j $, $\mu$ be Borel measures in $\Omega $ with $\mu_j (\Omega )\leq A$ for all $j$ and some $A>0$ and $\mu_j \rightharpoonup \mu $ as $j\to \infty$, that is $\int_\Omega f \d \mu_j \to \int_\Omega f \d \mu $ for all $f\in C_0 (\Omega )$. Let $g_j,g\in C(\partial \Omega )$ be such that $\| g_j -g \|_{C(\partial \Omega)}\to 0$ as $j\to \infty$. If the family of convex functions $\{ \phi_j \} \subset C(\overline{ \Omega } )$ satisfies
\[
\begin{cases}
M\phi_j = \mu_j \qquad &\text{ in } \Omega , \\
\phi_j = g_j &\text{ on } \partial \Omega , 
\end{cases}
\]
then $\{ \phi_j \}$ contains a subsequence $\{ \phi_{j_k} \}$, such that $ \phi_{j_k}\to \phi  $ uniformly on compact subsets of $\Omega$ as $k\to \infty$, where $\phi \in C(\overline{\Omega })$ is convex and $M\phi = \mu$ in $\Omega$, $\phi=g$ on $\partial \Omega$.
\end{lem}
The above lemma is proved in \cite{monge_ampere}, pp.\ 20-21, in the case $g_j\equiv g$, $j=1,2,\ldots $. The case $g_j\not \equiv g$ follows as a straightforward generalisation.
\section{Generalised comparison principle}\label{gen_comp_princ_section}
 For $\phi \in H^2 ({\Omega } )$ let 
\[\begin{split}
A_\phi &\coloneqq \{ x\in \Omega \, : \,  D^2\phi (x) \text{ is positive definite at } x \},\\
B_\phi &\coloneqq \{ x\in \Omega \, : \,  D^2\phi (x) \text{ is negative definite at } x \}.
\end{split}
\]
We will denote by $[M\phi ]^+$ the measure that is absolutely continuous with respect to the Lebesgue measure with density $(\mathrm{det}\, D^2 \phi )^+$. Note that, because $\phi \in H^2 (\Omega )$, the H\"older inequality gives $[M\phi ]^+ (\Omega ) < \infty$. Moreover, if $\phi \in C^2 (\Omega )$ then $A_\phi $ is an open subset of $\Omega$, $\phi $ is convex on $A_\phi$ and, using (i), $[M\phi ]^+$ is equal to the Monge--Amp\`ere measure $M\phi$ when restricted to $A_\phi$. 

We also we denote by $[M\phi ]^-$ the measure that is absolutely continuous with respect to the Lebesgue measure with density $(\mathrm{det}\, D^2 (-\phi) )^+$.
\begin{theorem}\label{gen_comp_princ}{\normalfont (Generalised comparison principle)}
Let $\Omega $ be a bounded, open set in $\mathbb{R}^n$. Let $\psi \in C(\overline{\Omega })$ be a convex function in $\Omega $ with $M\psi (\Omega )<\infty$ and let $\phi\in H^2 (\Omega ) $ be such that 
\begin{equation}\label{ma_measures_ineq}
 [  M\phi ]^+ \leq M\psi \quad \text{ in } \Omega.
\end{equation}
Then 
\[
\min_{\overline{\Omega }} (\phi-\psi ) = \min_{\partial \Omega } (\phi-\psi ) .
\]
In particular, if $\phi \geq \psi$ on $\partial \Omega$ then $\phi \geq \psi $ in $\overline{\Omega }$. Similarly, if $\psi$ is concave in $\Omega $ and $\phi\in H^{2} (\Omega )$ is such that $ [  M\phi ]^-\leq M(-\psi)$ in $\Omega$, then 
\[
\max_{\overline{\Omega }} (\phi-\psi ) = \max_{\partial {\Omega} } (\phi-\psi ) .
\]
\end{theorem}
We give a proof that does not use the solvability result of the Monge--Amp\`ere equation on $\Omega$, and so does not require strict convexity of $\Omega$ (see Theorem \ref{existence_thm}). Instead we replace it with the solvability result on a neighbourhood $B$ of a point in $\Omega$ and an application of the strong comparison principle (Theorem \ref{comparison_thm}). Since the resulting proof is therefore local in nature - it does not use any global properties of $\Omega$ - it allows for $\Omega$ to be nonconvex. (In fact the original proof due to \cite{rauch_taylor} does not use strict convexity of $\Omega$ when $\phi $ is assumed to be $ C^2( \Omega )$; but for $\phi \in H^2$ their approximation argument requires the solvability result (Theorem \ref{existence_thm}), which is only valid for $\Omega $ strictly convex.)
\begin{proof} We focus on the case of $\psi$ convex; the case of concave $\psi$ follows by replacing $\phi$, $\psi$ by $-\phi$, $-\psi$ respectively. Assume first that $\phi \in C^2 (\Omega )$ (here we can follow \cite{rauch_taylor}). Suppose otherwise that there exists $x_0 \in \Omega $ such that 
\[(\phi-\psi) (x_0) = \min_{x \in \overline{ \Omega }} (\phi-\psi )(x) \]
and consider the function 
\eqnb\label{def_of_psi_tilde}
\widetilde{\psi} (x) \coloneqq \psi(x) + \varepsilon_0 |x-x_0 |^2
\eqne
for $\varepsilon_0 >0$. Since $\Omega $ is bounded, it is clear that for $\varepsilon_0$ sufficiently small the function $\phi-\widetilde{\psi}$ still does not attain its minimum on $\partial \Omega $. This means that for such an $\varepsilon_0$ fixed there exists $\widetilde{x} \in \Omega$ such that 
\eqnb\label{x1}
(\phi-\widetilde{\psi} ) (\widetilde{x}) =  \min_{x \in \overline{ \Omega }} (\phi-\widetilde{\psi}  )(x) < \min_{x \in \partial { \Omega }} (\phi-\widetilde{\psi}  )(x).
\eqne
Moreover, because on $A_\phi$ both $\widetilde{\psi}$ and $\phi$ are convex and
\[
M\widetilde{\psi} \geq M\psi \geq M\phi,
\]
the strong comparison principle (Theorem \ref{comparison_thm}) gives $\widetilde{x} \not \in A_\phi$. In other words $D^2 \phi (\widetilde{x} )$ is not positive definite. Now, because any symmetric matrix is positive definite if and only if all its eigenvalues are positive (see e.g.\ Theorem 7.2.1 in \cite{horn_johnson}), we see that $D^2 \phi(\widetilde{x})$ has at least one nonpositive eigenvalue. Let $\lambda \leq 0$ be one such eigenvalue and let $\alpha \in \mathbb{R}^n$, $|\alpha |=1$, be the respective eigenvector. Then, by performing a Taylor expansion in the $\alpha$ direction, we can write, for $t\in \mathbb{R}$ with $|t|$ small
\begin{equation}\label{utaylor}
\phi(\widetilde{x} + t \alpha ) - \phi (\widetilde{x} ) = a_1 t  + \lambda t^2 + o(t^2),
\end{equation}
where $a_1  \in \mathbb{R}$ and $o( \cdot ) : \mathbb{R} \to \mathbb{R}$ denotes any function such that $o(y)/y \stackrel{y\to 0}{\longrightarrow } 0$. As $\psi$ is convex, it has a supporting hyperplane at $\widetilde{x}$ (see 4) ). Hence 
\[
\begin{split}
\widetilde{\psi} (\widetilde{x} + t \alpha ) - \widetilde{\psi} (\widetilde{x}) &= \psi(\widetilde{x} + t \alpha ) - \psi( \widetilde{x}) + \varepsilon_0 ( |\widetilde{x} + t\alpha - x_0 |^2 - |\widetilde{x} - x_0 |^2 )\\
&\geq  a_2 t + \varepsilon_0 ( |\widetilde{x}+ t\alpha - x_0 |^2 - |\widetilde{x} - x_0 |^2 )= a_3 t + \varepsilon_0 \, t^2, 
\end{split}
\]
where $a_2,a_3 \in \mathbb{R}$. Combining this with (\ref{utaylor}) and using \eqref{x1}, we obtain
\[
(\phi-\widetilde{\psi} ) (\widetilde{x})  {\leq } (\phi-\widetilde{\psi} ) (\widetilde{x}+ t \alpha )\leq  (\phi-\widetilde{\psi} ) (\widetilde{x} ) +(a_1-a_3)t+ (\lambda - \varepsilon_0 ) t^2 + o(t^2)
\]
for small values of $|t|$. This means that the quadratic polynomial \[(a_1-a_3)t + (\lambda - \varepsilon_0 ) t^2 \] attains its minimum at $t=0$. Hence $a_1=a_3$ and $\lambda - \varepsilon_0 \geq 0$, which contradicts $\lambda \leq 0< \varepsilon_0$.\\

Now let $\phi \in H^{2} (\Omega )$, and similarly as before consider $\widetilde{\psi }$ and $\widetilde{x}$ given by \eqref{def_of_psi_tilde} and \eqref{x1}. Let $B$ be an open ball centered at $\widetilde{x}$ and such that $\overline{B}\subset \Omega $. 

Let $\{ \phi_j \} \subset C_0^\infty (\mathbb{R}^2)$ be such that $\| \phi_j - \phi \|_{H^2 (B)}\to 0$ as $j\to \infty$. By the embedding $H^2 (B ) \subset C^0 (\overline{B } )$ we also have $\| \phi_j - \phi \|_{C^0 (\overline{B})}\to 0$ as $j\to \infty$. Let $\mu_j$, $\mu$ be Borel measures on $B$ defined by $\mu_j := [M\phi_j]^+$, $\mu := [M\phi]^+$ (note that $\mu_j (\Omega), \mu (\Omega) < \infty$ due to the H\"{o}lder inequality). For each $j$ let $\psi_j$ be the unique convex solution of the Dirichlet problem
\[
\begin{cases}
M\psi_j = \mu_j \qquad &\text{ in } B , \\
\psi_j = \phi_j &\text{ on } \partial B . 
\end{cases}
\]
The existence of such $\psi_j$ is guaranteed by the existence theorem (Theorem \ref{existence_thm}). Because $\phi_j \in C^2$, the first part gives 
\begin{equation}\label{ineq_on_ball}
\psi_j \leq \phi_j \quad \text{ in }\overline{B}.
\end{equation}
Furthermore, because $\| (\mathrm{det}\, D^2 \phi_j )^+-(\mathrm{det}\, D^2 \phi )^+ \|_{L^1(B)}\to 0$ as $j\to \infty$ gives $\mu_j \rightharpoonup \mu$, and because $\| \phi_j - \phi \|_{C^0 (\partial {B})}\to 0$, we can use the convergence lemma (Lemma~\ref{conv_lemma_2}) to obtain  that $\psi_j \to \Psi$ uniformly on compact subsets of $B$ for some subsequence (which we relabel), where $\Psi\in\nobreak C^0 (\overline{B})$ is convex and satisfies
\[
\begin{cases}
M\Psi = \mu \qquad &\text{ in } B , \\
\Psi = \phi &\text{ on } \partial B . 
\end{cases}
\]
Taking the limit $j\to \infty$ in \eqref{ineq_on_ball} we get $\Psi \leq \phi$ on $\overline{B}$ and so in particular $\Psi (\widetilde{x})\leq \phi (\widetilde{x})$ and
\eqnb\label{temp33}
(\Psi - \widetilde{\psi }) (\widetilde{x}) \leq (\phi - \widetilde{\psi }) (\widetilde{x}) = \min_{\overline{\Omega}} (\phi - \widetilde{\psi }) \leq \min_{\partial B } (\phi - \widetilde{\psi }) = \min_{\partial B } (\Psi - \widetilde{\psi }) .
\eqne
Because $M\Psi = \mu  = [M\phi ]^+ \leq M\psi \leq M\widetilde{\psi}$ on $B$ and both $\Psi$ and $\widetilde{\psi}$ are convex we can use the comparison principle (Corollary \ref{comparison_cor}) to write $\min_{\partial B } (\Psi - \widetilde{\psi })=\min_{\overline{B} } (\Psi - \widetilde{\psi })$. Therefore \eqref{temp33} becomes
\[
(\Psi - \widetilde{\psi }) (\widetilde{x}) \leq \min_{\overline{B} } (\Psi - \widetilde{\psi }),
\]
that is $\Psi - \widetilde{ \psi }$ admits an internal minimum in $B$. This contradicts the strong comparison principle (Theorem \ref{comparison_thm}).
\end{proof}
An immediate consequence of the generalised comparison principle is that a solution to the Monge--Amp\`ere equation with sign-changing right-hand side can be bounded above and below by, respectively, the concave and the convex solutions of certain Monge--Amp\`ere problems.
\begin{cor}\label{mainthm}
Let $\Omega $ be a bounded, open subset of $\mathbb{R}^n$. If $\phi \in H^2 (\Omega )$, $\Phi_{conv}$ is a convex generalised solution to $\mathrm{det}\, D^2 \Phi_{conv} = \left( \mathrm{det}\, D^2 \phi \right)^+$ and $\Phi_{conc}$ is a concave generalised solution to $\mathrm{det}\, D^2 (-\Phi_{conc}) = \left( \mathrm{det}\, D^2 (-\phi ) \right)^+$ such that $ \Phi_{conv} = \Phi_{conc} = \phi$ on $\p \Omega$, then
\[
\Phi_{conv} \leq \phi \leq \Phi_{conc} \quad \text{ in }\overline{\Omega }.
\]
\end{cor}
\begin{proof}
This follows from the generalised comparison principle (Theorem \ref{gen_comp_princ}) since $M\phi \leq M\Phi_{conv}= M(-\Phi_{conc}) $ and 
\[ M\Phi_{conv} (\Omega )= M(-\Phi_{conc})(\Omega ) = \| \left( \mathrm{det}\, D^2 (-\phi ) \right)^+ \|_{L^1} \leq C \| \phi \|_{H^2} < \infty. \qedhere \]
\end{proof}
Note that if $\Omega $ is strictly convex then the functions $\Phi_{conv}$, $\Phi_{conc}$ are uniquely determined by the existence theorem (Theorem \ref{existence_thm}). What is more, if $n$ is even then $\mathrm{det} \, D^2 (-\phi )= \mathrm{det} \, D^2 \phi  $ and hence $\Phi_{conv}$ and $\Phi_{conc}$ are solutions to the same problem
\[
\begin{cases}
\mathrm{det}\, D^2 \Phi = \left( \mathrm{det}\, D^2 \phi \right)^+ \quad & \text{ in } \Omega , \\
\Phi=\phi &\text{ on } \partial \Omega .
\end{cases}
\] 
In other words, if $n$ is even, then any $\phi \in H^2(\Omega )$ can be bounded below and above using functions $\Phi_{conv}$ and $\Phi_{conc}$ which depend only on the positive part of $\mathrm{det}\, D^2 \phi$ and the boundary values of $\phi$. The power of Corollary \ref{mainthm} is demonstrated by the following nonexistence result.
\begin{cor}\label{main_result}
Let $n$ be even, $\Omega $ a bounded, open subset of $\mathbb{R}^n$, $C\in \mathbb{R}$ and $f$ a nonpositive function such that $f\not \equiv 0$. Then the problem
\[
\begin{cases}
\mathrm{det}\, D^2 \phi = f\qquad & \text{ in } \Omega , \\
\phi=C &\text{ on } \partial \Omega 
\end{cases}
\]
has no $H^2 (\Omega ) $ solution.
\end{cor}
\begin{proof}
Suppose there exists $\phi \in H^2 ({\Omega })$, a solution of the above problem. The constant function $\Phi \equiv C$ satisfies $\mathrm{det}\,D^2 \Phi = 0 = f^+$ with $\left. \Phi \right|_{\partial \Omega } = C$. Therefore, by Corollary~\ref{mainthm}, $C \leq \phi \leq C$, i.e. $\phi\equiv C$. Hence $0 \equiv \mathrm{det} \, D^2 \phi \equiv f  \not \equiv 0$, a contradiction.
\end{proof}
\section{An application to the 2D Navier--Stokes equations}\label{application_NS_section}
Let us consider the two-dimensional Navier--Stokes equations
\[
u_t + (u \cdot \nabla ) u -  \Delta u + \nabla p =  0
\]
at any $t>0$ equipped with incompressibility constraint $\mathrm{div}\, u =0$. Taking the divergence of the equations and using the incompressibility constraint we obtain
\[
\nabla \cdot \left[ (u \cdot \nabla ) u\right]  +\Delta p =0 .
\]
Now, because any divergence-free $2$D vector field can be represented as $u=(\phi_y,-\phi_x)$ for some scalar function $\phi$, we can write
\[\begin{split}
-\Delta p &= \p_x (u_2 \p_x u_1 + u_1 \p_x u_1) + \p_y (u_2 \p_y u_2 + u_1 \p_x u_2 ) \\ 
&= \p_x (-\phi_x \phi_{yy} +\phi_y \phi_{xy} ) + \p_y (\phi_x \phi_{xy} - \phi_y \phi_{xx}) = -2 \phi_{xx} \phi_{yy} + 2 (\phi_{xy})^2,
\end{split}
\]
that is
\begin{equation}\label{preMA}
 \phi_{xx}  \phi_{yy} - (\phi_{xy})^2 = \frac{1}{2} \Delta p .
\end{equation}
This is the Monge--Amp\`{e}re equation
\begin{equation}\label{MA}
\mathrm{det}\, D^2\phi=\frac{1}{2}\Delta p .
\end{equation}
This connection between the pressure $p$ and velocity $u$ in $2$D Navier--Stokes equations was first studied by Larchev\^{e}que (1990 \& 1993), who also observed that in the regions of positive $\Delta p$ the velocity $u$ has closed streamlines, which he related to the appearance of {coherent structures}. In contrast to this local analysis, here we use the results of the previous section to show that if $\Delta p \not \equiv 0$ then it is not possible that $\Delta p \leq 0$ throughout $\Omega$. Indeed, because the global-in-time solution $(u,p)$ of the two-dimensional Navier-Stokes equations is smooth, given $C^1$ regularity of $\p \Omega$, we have in particular that $u\in H^1_0 (\Omega ) $, that is $\phi \in H^2 (\Omega )$ and $\nabla \phi = 0 $ on $\partial \Omega  $. In particular $\left.  \phi \right|_{\partial \Omega}=C$ for some $C\in \mathbb{R}$ and from the last corollary we obtain that $\Delta p \leq 0$ (with $\Delta p  \ne 0$) cannot hold throughout $\Omega$.\\

We also note that $\Delta p>0$ cannot throughout $\Omega$, which can be shown using elementary methods. Indeed, because the solution $(u,p)$ to the $2$D Navier--Stokes equations is smooth (see e.g. \cite{lions_prodi}, Section 3.3 of \cite{temam} or Section 9.6 of \cite{robinson_red_book}) we have in particular that $u\in C^1(\Omega )$, that is $\phi \in C^2 (\Omega )$. Therefore, if $\Delta p>0$ we can follow an idea from Section IV.6.3 of \cite{courant_hilbert} to write, using \eqref{preMA},
\[\phi_{xx} \phi_{yy} \geq  \mathrm{det}\, D^2 \phi = \frac{1}{2} \Delta p >0\]
and we see (by continuity) that either
\begin{equation}\label{samesign}
\phi_{xx}, \phi_{yy} >0 \text{ in }\Omega \quad \text{ or } \quad \phi_{xx}, \phi_{yy} <0 \text{ in }\Omega .
\end{equation}
Supposing that $\phi_{xx}, \phi_{yy} >0$, we can use the divergence theorem to obtain
\[
0< \int_\Omega  \Delta \phi \,\mathrm{d}x \,\mathrm{d} y = \int_{\partial \Omega } \frac{\partial \phi }{\partial \nu } \mathrm{d}S =0,
\]
a contradiction; we argue similarly if $\phi_{xx}, \phi_{yy} <0$. 

Therefore, if at any time $t>0$ we have $\Delta p \not \equiv 0$ then either $\Delta p $ changes sign inside the domain or $\Delta p \geq 0$ with $\Delta p \not > 0$. In either case $\Delta p=0$ at some interior point of the domain.

One of the questions related to the connection of the pressure $p$ and velocity $u$ in the $2$D incompressible Navier--Stokes equations is whether the pressure determines the velocity uniquely (see the review article \cite{robinson_2D_NSE}). The answer to this question is negative, as the following example shows. \vspace{0.5cm} \\
\textbf{Example} Consider the shear flow $u(x,y,t) = (U(y,t),0)$ in a channel $\Omega \coloneqq \mathbb{T} \times [0,1]$, where $U$ satisfies the $1$D heat equation $U_t-U_{yy}=0$ in $[0,1]\times [0,\infty)$ with boundary conditions $U(0,t)=U(1,t)=0$. Note that $U(y,t):=C\mathrm{e}^{-k^2t} \sin (kx)$ is a solution of this problem for any $C\ne 0$, $k\in \mathbb{N}$. Then the pair $(u,p)$, where $p\equiv 0$, satisfies the $2$D incompressible Navier--Stokes equations as $\mathrm{div} \, u$ vanishes and $u_t-\Delta u + (u \cdot \nabla )u +\nabla p = u_t - u_{yy}=0$.\\

This example also illustrates the relevance of boundary conditions in Corollary \ref{mainthm}. Indeed, if the periodic boundary condition (in $x$) was replaced by the homogeneous Dirichlet boundary condition, then Corollary \ref{mainthm} implies that the only velocity field $u$ corresponding to $p\equiv 0$ is $u\equiv 0$.

\section{Acknowledgements}
I would like to thank James Robinson for his support and for the suggestion of the shear flow example in Section \ref{application_NS_section}. This research is supported by EPSRC as part of the MASDOC DTC at the University of Warwick, Grant No. EP/HO23364/1.
\thispagestyle{empty}
\bibliography{liter}
\thispagestyle{empty}
\end{document}